\documentclass[10pt]{amsart}
\usepackage{latexsym,amssymb,amsmath,amstext,amscd,mathrsfs}
\usepackage{graphicx,verbatim,amsfonts,indentfirst,enumerate,txfonts}
\usepackage[all]{xy}
\usepackage{pxfonts}
\usepackage{bezier,longtable}
\usepackage[labelformat=empty]{caption}

\newtheorem{theo}{Theorem}[section]
\newtheorem{lem}[theo]{Lemma}

\newtheorem{cor}[theo]{Corollary}
\newtheorem{prop}[theo]{Proposition}

\newtheorem{conject}[theo]{Conjecture}
\renewenvironment{proof}{ \emph{Proof}}{$\Box$}
\newtheorem{defi}[theo]{Definition}

\newcommand{\supp}{\mathrm{supp}}

\newcommand{\ch}{\mathrm{ch}}
\newcommand{\Soc}{\mathrm{Soc}}
\newcommand{\pro}{\mathrm{pro}}
\newcommand{\ind}{\mathrm{ind}}
\newcommand{\Hom}{\mathrm{Hom}}

\newcommand{\rk}{\mathrm{rk}}
\newcommand{\field}[1]{\mathbb{#1}}

\newcommand{\bb}{\mathfrak{b}}

\renewcommand{\gg}{\mathfrak{g}}
\newcommand{\hh}{\mathfrak{h}}
\newcommand{\kk}{\mathfrak{k}}
\renewcommand{\ll}{\mathfrak{l}}
\newcommand{\mm}{\mathfrak{m}}
\newcommand{\nn}{\mathfrak{n}}
\newcommand{\pp}{\mathfrak{p}}
\newcommand{\qq}{\mathfrak{q}}

\renewcommand{\ss}{\mathfrak{s}}
\renewcommand{\sp}{\mathfrak{sp}}
\newcommand{\so}{\mathfrak{so}}
\renewcommand{\tt}{\mathfrak{t}}
\renewcommand{\sl}{\mathfrak{sl}}

\newcommand{\ZZ}{\mathbb{Z}}

\def\cplus{\hbox{$\subset${\raise0.3ex\hbox{\kern -0.55em ${\scriptscriptstyle +}$}}\ }}
\def\clplus{\hbox{$\subset${\raise0.3ex\hbox{\kern -0.55em ${\scriptscriptstyle +}$}}\ }}
\def\crplus{\hbox{$\supset${\raise1.05pt\hbox{\kern -0.55em ${\scriptscriptstyle +}$}}\ }}

\title[Algebraic methods in the theory of generalized Harish-Chandra modules]{Algebraic methods in the theory of generalized Harish-Chandra modules}
\author[Ivan Penkov]{\;Ivan~Penkov}
\address{
Ivan Penkov 
\newline School of Engineering and Science
\newline Jacobs University Bremen
\newline Campus Ring 1, Bremen
\newline 28759, Bremen, Germany}
\email{i.penkov@jacobs-university.de}

\author[Gregg Zuckerman]{\;Gregg Zuckerman}
\address{
Gregg Zuckerman
\newline Department of Mathematics
\newline Yale University
\newline 10 Hillhouse Avenue, P.O. Box 208283
\newline New Haven, CT 06520-8283, USA}
\email{gregg.zuckerman@yale.edu}

\begin{document}

\begin{abstract}
This paper is a review of results on generalized Harish-Chandra modules in the framework of cohomological induction. The main results, obtained during the last 10 years, concern the structure of the fundamental series of $(\gg,\kk)-$modules, where $\gg$ is a semisimple Lie algebra and $\kk$ is an arbitrary algebraic reductive in $\gg$ subalgebra. These results lead to a classification of simple $(\gg,\kk)-$modules of finite type with generic minimal $\kk-$types, which we state. We establish a new result about the Fernando-Kac subalgebra of a fundamental series module. In addition, we pay special attention to the case when $\kk$ is an eligible $r-$subalgebra (see the definition in section \ref{sect:4}) in which we prove stronger versions of our main results. If $\kk$ is eligible, the fundamental series of $(\gg,\kk)-$modules yields a natural algebraic generalization of Harish-Chandra's discrete series modules. 

\textbf{Mathematics Subject Classification (2010).} Primary 17B10, 17B55. 

\textbf{Key words:} generalized Harish-Chandra module, $(\gg,\kk)-$module of finite type, minimal $\kk-$type, Fernando-Kac subalgebra, eligible subalgebra.
\end{abstract}

\maketitle

\section*{Introduction}\label{sect:intro}
Generalized Harish-Chandra modules have now been actively studied for more than 10 years. A \textit{generalized Harish-Chandra module} $M$ over a finite-dimensional reductive Lie algebra $\gg$ is a $\gg-$module $M$ for which there is a reductive in $\gg$ subalgebra $\kk$ such that as a $\kk-$module, $M$ is the direct sum of finite-dimensional generalized $\kk-$isotypic components. If $M$ is irreducible, $\kk$ acts necessarily semisimply on $M$, and in what follows we restrict ourselves to the study of generalized Harish-Chandra modules on which $\kk$ acts semisimply; see \cite{Z} for an introduction to the topic.

In this paper we present a brief review of results obtained in the past 10 years in the framework of algebraic representation theory, more specifically in the framework of cohomological induction, see \cite{KV} and \cite{Z}. In fact, generalized Harish-Chandra modules have been studied also with geometric methods, see for instance \cite{PSZ} and \cite{PS1}, \cite{PS2}, \cite{PS3}, \cite{Pe}, but the geometric point of view remains beyond the scope of the current review. In addition, we restrict ourselves to finite-dimensional Lie algebras $\gg$ and do not review the paper \cite{PZ4}, which deals with the case of locally finite Lie algebras. We omit the proofs of most results which have already appeared.

The cornerstone of the algebraic theory of generalized Harish-Chandra modules so far is our work \cite{PZ2}. In this work we define the notion of simple generalized Harish-Chandra modules with generic minimal $\kk-$type and provide a classification of such modules. The result extends in part the Vogan-Zuckerman classification of simple Harish-Chandra modules. It leaves open the questions of existence and classification of simple $(\gg,\kk)-$modules of finite type whose minimal $\kk-$types are not generic. While the classification of such modules presents the main open problem in the theory of generalized Harish-Chandra modules, in the note \cite{PZ3} we establish the existence of simple $(\gg,\kk)-$modules with arbitrary given minimal $\kk-$type.

In the paper \cite{PZ5} we establish another general result, namely the fact that each module in the fundamental series of generalized Harish-Chandra modules has finite length. We then consider in detail the case when $\kk=\sl(2)$. In this case the highest weights of $\kk-$types are just non-negative integers $\mu$ and the genericity condition is the inequality $\mu \geq \Gamma$, $\Gamma$ being a bound depending on the pair $(\gg,\kk)$. In \cite{PZ5} we improve the bound $\Gamma$ to an, in general, much lower bound $\Lambda$. Moreover, we show that in a number of low dimensional examples the bound $\Lambda$ is sharp in the sense that the our classification results do not hold for simple $(\gg,\kk)-$ modules with minimal $\kk-$type $V(\mu)$ for $\mu$ lower than $\Lambda$. In \cite{PZ5} we also conjecture that the Zuckerman functor establishes an equivalence of a certain subcategory of the thickening of category $O$ and a subcategory of the category of $(\gg,\kk \simeq \sl(2))-$modules. 

Sections \ref{sect:2} and \ref{sect:3} of the present paper are devoted to a brief review of the above results. We also establish some new results in terms of the algebra $\tilde{\kk} := \kk+C(\kk)$ (where $C(\cdot)$ stands for centralizer in $\gg$). A notable such result is Corollary \ref{corr:6} which gives a sufficient condition on a simple $(\gg,\kk)-$module $M$ for $\tilde{\kk}$ to be a maximal reductive subalgebra of $\gg$ which acts locally finitely on $M$. 

The idea of bringing $\tilde{\kk}$ into the picture leads naturally to considering a preferred class of reductive subalgebras $\kk$ which we call eligible: they satisfy the condition $C(\tt) = \tt+C(\kk)$ where $\tt$ is Cartan subalgebra of $\kk$. In section \ref{sect:5} we study a natural generalization of Harish-Chandra's discrete series to the case of an eligible subalgebra $\kk$. A key statement here is that under the assumption of eligibility of $\kk$, the isotypic component of the minimal $\kk-$type of a generalized discrete series module is an irreducible $\tilde{\kk}-$module (Theorem \ref{th3}).
 
\section*{Acknowledgements}\label{sect:acknow.}
I. Penkov thanks Yale University for its hospitality and partial financial support during the spring of 2012 when this paper was conceived, as well as the Max Planck Institute for Mathematics in Bonn where the work on the paper was continued. G. Zuckerman thanks Jacobs University for its hospitality. Both authors have been partially supported by the DFG through Priority Program $1388$ ``Representation theory''.
 
\vspace{2cm}
 
\section{Notation and preliminary results}\label{sect:1}
We start by recalling the setup of \cite{PZ2} and \cite{PZ5}.
\subsection{Conventions}\label{sect:1.1}

The ground field is $\field{C}$, and if not explicitly stated otherwise, all vector spaces and Lie algebras are defined over $\field{C}$. The sign $\otimes$ denotes tensor product over $\field{C}$. The superscript $^*$ indicates dual space. The sign $\clplus$ stands for semidirect sum of Lie algebras (if $\ll=\ll '\clplus\ll ''$, then $\ll '$ is an ideal in $\ll$ and $\ll ''\cong\ll /\ll '$). $H^\cdot(\ll,M)$ stands for the cohomology of a Lie algebra $\ll$ with coefficients in an $\ll$-module $M$, and $M^\ll=H^0(\ll,M)$ stands for space of $\ll$-invariants of $M$. By $Z(\ll)$ we denote the center of $\ll$, and by $\ll_{ss}$ we denote the semisimple part of $\ll$ when $\ll$ is reductive. $\Lambda^\cdot(\cdot)$ and $S^\cdot(\cdot)$ denote respectively the exterior and symmetric algebra.

If $\ll$ is a Lie algebra, then $U(\ll)$ stands for the enveloping algebra of $\ll$ and $Z_{U(\ll)}$ denotes the center of $U(\ll)$. We identify $\ll$-modules with $U(\ll)$-modules. It is well known that if $\ll$ is finite dimensional and $M$ is a simple $\ll$-module (or equivalently a simple $U(\ll)$-module), $Z_{U(\ll)}$ acts on $M$ via a $Z_{U(\ll)}$-\textit{character}, i.e. via an algebra homomorphism $\theta_{M}: Z_{U(\ll)}\rightarrow\field{C}$, see Proposition 2.6.8 in \cite{Dix}.

We say that an $\ll$-module $M$ is \textit{generated} by a subspace $M'\subseteq M$ if $U(\ll)\cdot M'=M$, and we say that $M$ is \textit{cogenerated} by $M'\subseteq M$, if for any non-zero homomorphism $\psi :M\rightarrow\bar{M}$, $M'\cap\ker\psi\neq\{0\}$. 

By $\Soc M$ we denote the socle (i.e. the unique maximal semisimple submodule) of an $\ll$-module $M$. If $\omega\in\ll^*$, we put $M^\omega:=\{m\in M\; |\; \ll\cdot m=\omega(\ll)m \;\forall l \in\ll\}$. By $\supp_\ll M$ we denote the set $\{\omega\in\ll^*\; |\; M^\omega\neq 0\}$.

A finite \textit{multiset} is a function $f$ from a finite set $D$ into $\field{N}$. A \textit{submultiset} of $f$ is a multiset $f'$ defined on the same domain $D$ such that $f'(d)\leq f(d)$ for any $d\in D$. For any finite multiset $f$, defined on a subset $D$ of a vector space, we put $\rho_f:=\frac{1}{2}\sum_{d\in D}f(d)d$.

If $\dim M<\infty$ and $M=\bigoplus_{\omega\in\ll^*}M^\omega$, then $M$ determines the finite multiset $\ch_{\ll}M$ which is the function $\omega\mapsto\dim M^\omega$ defined on $\supp_{\ll}M$.

\subsection{Reductive subalgebras, compatible parabolics and generic $\kk$-types}\label{sect:1.2}

Let $\gg$ be a finite-dimensional semisimple Lie algebra. By $\gg$-mod we denote the category of $\gg$-modules. Let $\kk\subset\gg$ be an algebraic subalgebra which is reductive in $\gg$. We set $\tilde{\kk} = \kk + C(\kk)$ and note that $\tilde{\kk}=\kk_{ss} \oplus C(\kk)$ where $C(\cdot)$ stands for centralizer in $\gg$. We fix a Cartan subalgebra $\tt$ of $\kk$ and let $\hh$ denote an as yet unspecified Cartan subalgebra of $\gg$. Everywhere, but in subsection \ref{sect:1.3} below, we assume that $\tt \subseteq \hh$, and hence that $\hh \subseteq C(\tt)$. By $\Delta$ we denote the set of $\hh$-roots of $\gg$, i.e. $\Delta=\{\supp_\hh\gg\}\setminus\{0\}$. Note that, since $\kk$ is reductive in $\gg$, $\gg$ is a $\tt$-weight module, i.e. $\gg=\bigoplus_{\eta\in\tt^*}\gg^\eta$. We set $\Delta_\tt:=\{\supp_\tt\gg\}\setminus\{0\}$. Note also that the $\field{R}$-span of the roots of $\hh$ in $\gg$ fixes a real structure on $\hh^*$, whose projection onto $\tt^*$ is a well-defined real structure on $\tt^*$. In 
what follows, 
we 
denote by $\mathrm{Re}\eta$ the real part of an 
element $\eta\in\tt^*$. We fix also a Borel subalgebra $\bb_\kk\subseteq\kk$ with $\bb_\kk\supseteq\tt$. Then $\bb_\kk=\tt\crplus\nn_\kk$, where $\nn_\kk$ is the nilradical of $\bb_\kk$. We set $\rho:=\rho_{\ch_\tt\nn_\kk}$. The quartet $\gg,\kk,\bb_\kk,\tt$ will be fixed throughout the paper. By $W$ we denote the Weyl group of $\gg$.

As usual, we parametrize the characters of $Z_{U(\gg)}$ via the Harish-Chandra homomorphism. More precisely, if $\bb$ is a given Borel subalgebra of $\gg$ with $\bb\supset\hh$ ($\bb$ will be specified below), the $Z_{U(\gg)}$-character corresponding to $\zeta\in\hh^*$ via the Harish-Chandra homomorphism defined by $\bb$ is denoted by $\theta_\zeta$ ($\theta_{\rho_{\ch_\hh\bb}}$ is the trivial $Z_{U(\gg)}$-character). Sometimes we consider a reductive subalgebra $\ll \subset \gg$ instead of $\gg$ and apply this convention to the characters of $Z_{U(\ll)}$. In this case we write $\theta_{\zeta}^{\ll}$ for $\zeta \in \hh_{\ll}^*$, where $\hh_{\ll}$ is a Cartan subalgebra of $\ll$. 

By $\langle\cdot\,\,\cdot\rangle$ we denote the unique $\gg$-invariant symmetric bilinear form on $\gg^*$ such that $\langle\alpha,\alpha\rangle=2$ for any long root of a simple component of $\gg$. The form $\langle\cdot\,,\,\cdot\rangle$ enables us to identify $\gg$ with $\gg^*$. Then $\hh$ is identified with $\hh^*$, and $\kk$ is identified with $\kk^*$. We sometimes consider $\langle\cdot\,,\,\cdot\rangle$ as a form on $\gg$. The superscript $\perp$ indicates orthogonal space. Note that there is a canonical $\kk$-module decomposition $\gg=\kk\oplus\kk^\perp$ and a canonical decomposition $\hh=\tt\oplus\tt^\perp$ with $\tt^\perp\subseteq\kk^\perp$. We also set $\parallel\zeta\parallel^2:=\langle\zeta,\zeta\rangle$ for any $\zeta\in\hh^*$.

We say that an element $\eta\in\tt^*$ is $(\gg,\kk)$-\textit{regular} if $\langle\mathrm{Re}\eta,\sigma\rangle\neq 0$ for all $\sigma\in\Delta_\tt$. To any $\eta\in\tt^*$ we associate the following parabolic subalgebra $\pp_\eta$ of $\gg$: $$\pp_\eta=\hh\oplus(\bigoplus_{\alpha\in\Delta_\eta}\gg^\alpha),$$ where $\Delta_\eta:=\{\alpha\in\Delta\; |\; \langle\mathrm{Re}\eta,\alpha\rangle\geq 0\}$. By $\mm_\eta$ and $\nn_\eta$ we denote respectively the reductive part of $\pp$ (containing $\hh$) and the nilradical of $\pp$. In particular $\pp_\eta=\mm_\eta\crplus\nn_\eta$, and if $\eta$ is $\bb_\kk$-dominant, then $\pp_\eta\cap\kk=\bb_\kk$. We call $\pp_\eta$ a \textit{$\tt$-compatible parabolic subalgebra}. Note that $$\pp_\eta = C(\tt)\oplus(\bigoplus_{\beta \in \Delta^+_{\tt,\eta}}\,\gg^\beta)$$ where $\Delta^+_{\tt,\eta} := \lbrace\beta\in\Delta_\tt\,|\,\langle \mathrm{Re}\eta,\beta\rangle\,>\,0\rbrace$. Hence, $\pp_\eta$ depends upon our choice of $\tt$ and $\eta$, but not upon the choice of $\hh$.  

A $\tt$-compatible parabolic subalgebra $\pp=\mm\crplus\nn$ (i.e. $\pp=\pp_\eta$ for some $\eta\in\tt^*$) is $\tt$-\textit{minimal} (or simply \textit{minimal}) if it does not properly contain another $\tt$-compatible parabolic subalgebra. It is an important observation that if $\pp=\mm\crplus\nn$ is minimal, then $\tt\subseteq Z(\mm)$. In fact, a $\tt$-compatible parabolic subalgebra $\pp$ is minimal if and only if $\mm$ equals the centralizer $C(\tt)$ of $\tt$ in $\gg$, or equivalently if and only if $\pp=\pp_\eta$ for a $(\gg,\kk)$-regular $\eta \in \tt^*$. In this case $\nn\cap\kk =\nn_\kk$.

Any $\tt$-compatible parabolic subalgebra $\pp =\pp_\eta$ has a well-defined opposite parabolic subalgebra $\bar{\pp}:=\pp_{-\eta}$; clearly $\pp$ is minimal if and only if $\bar{\pp}$ is minimal.

A \textit{$\kk$-type} is by definition a simple finite-dimensional $\kk$-module. By $V(\mu)$ we denote a $\kk$-type with $\bb_\kk$-highest weight $\mu$. The weight $\mu$ is then $\kk$-integral (or, equivalently, $\kk_{ss}-$integral) and $\bb_\kk$-dominant.

Let $V(\mu)$ be a $\kk$-type such that $\mu+2\rho$ is $(\gg,\kk)$-regular, and let $\pp=\mm\crplus\nn$ be the minimal compatible parabolic subalgebra $\pp_{\mu+2\rho}$. Put $\tilde{\rho}_\nn :=\rho_{\ch_\hh \nn}$ and $\rho_\nn:=\rho_{\ch_\tt\nn}$. Clearly $\rho_\nn =\tilde{\rho}_\nn |_\tt$. We define $V(\mu)$ to be \textit{generic} if the following two conditions hold:
\begin{enumerate}
\item $\langle\mathrm{Re}\mu+2\rho-\rho_\nn,\alpha\rangle\geq 0\;\forall\alpha\in\supp_\tt\nn_\kk$;
\item $\langle\mathrm{Re}\mu+2\rho-\rho_S,\rho_S\rangle >0$ for every submultiset $S$ of $\ch_\tt\nn$.
\end{enumerate}

It is easy to show that there exists a positive constant $C$ depending only on $\gg,\kk$ and $\pp$ such that $\langle\mathrm{Re}\mu+2\rho,\alpha\rangle >C$ for every $\alpha\in\supp_\tt\nn$ implies $\pp_{\mu+2\rho}=\pp$ and that $V(\mu)$ is generic.

\subsection{Generalities on $\gg-$modules}\label{sect:1.3}

Suppose $M$ is a $\gg-$module and $\ll$ is a reductive subalgebra of $\gg$. $M$ is \textit{locally finite over $Z_{U(\ll)}$} if every vector in $M$ generates a finite-dimensional $Z_{U(\ll)}-$module. Denote by $\mathcal{M}(\gg,Z_{U(\ll)})$ the full subcategory of $\gg-$modules which are locally finite over $Z_{U(\ll)}$. 

Suppose $M \in \mathcal{M}(\gg,Z_{U(\ll)})$ and $\theta$ is a $Z_{U(\ll)}-$character. Denote by $P(\ll, \theta)(M)$ the generalized $\theta-$eigenspace of the restriction of $M$ to $\ll$. The $Z_{U(\ll)}-$\textit{spectrum} of $M$ is the set of characters $\theta$ of $Z_{U(\ll)}$ such that $P(\ll, \theta)(M) \neq 0$. Denote the $Z_{U(\ll)}$ spectrum of $M$ by $\sigma(\ll,M)$. We say that $\theta$ is a \textit{central character of $\ll$ in $M$} if $\theta \in \sigma(\ll,M)$. The following is a standard fact.

\begin{lem}\label{lem:1}
 If $M \in \mathcal{M}(\gg,Z_{U(\ll)})$, then 
 $$M = \bigoplus_{\theta \in \sigma(\ll,M)} P(\ll,\theta)(M).$$
\end{lem}

A $\gg-$module $M$ is \textit{locally Artinian over} $\ll$ if for every vector $v \in M$, $U(\ll)\cdot v$ is an $\ll-$module of finite length.

\begin{lem}\label{lem:2}
 If $M$ is locally Artinian over $\ll$, then $M \in \mathcal{M}(\gg, Z_{U(\ll)})$.
\end{lem}

\begin{proof}
The statement follows from the fact that $Z_{U(\ll)}$ acts via a character on any simple $\ll-$module.
\end{proof}

If $\pp$ is a parabolic subalgebra of $\gg$, by a $(\gg,\pp)-$\textit{module} $M$ we mean a $\gg-$module $M$ on which $\pp$ acts locally finitely. By $\mathcal{M}(\gg,\pp)$ we denote the full subcategory of $\gg-$modules which are $(\gg,\pp)-$modules.

In the remainder of this subsection we assume that $\hh$ is a Cartan subalgebra of $\gg$ such that $\hh_\ll :=  \hh \cap \ll$ is a Cartan subalgebra of $\ll$, and that $\pp$ is a parabolic subalgebra of $\gg$ such that $\hh \subset \pp$ and $\pp \cap \ll$ is a parabolic subalgebra of $\ll$. By $M$ we denote a $\gg-$module from $\mathcal{M}(\gg,\pp)$.

\begin{lem}\label{lemma:3}
 The set $\supp_\hh\, M$ is independent of the choice of $\hh \subseteq \pp$.
\end{lem}

\begin{proof}
 Suppose $\hh_1$ and $\hh_2$ are Cartan subalgebras of $\gg$ such that $\hh_1, \hh_2 \subseteq \pp$. Let $\mm_j$ be the maximal reductive subalgebra of $\pp$ such that $\hh_j \subseteq \mm_j, j=1,2$. There exits an inner automorphism $\Psi(\mm_1) = \mm_2$. Then, $\Psi(\hh_1)$ and $\hh_2$ are Cartan subalgebras of $\mm_2$. There exists an inner automorphism $\Phi$ of $\mm_2$ such that $\Phi(\Psi(\hh_1)) = \hh_2$. Hence, for any finite dimensional $\pp-$module $W$, $\supp_{\hh_1}\,W = \supp_{\hh_2}\,W$. By assumption $M$ is a union of finite-dimensional $\pp-$modules. 
\end{proof}

\begin{prop}\label{prop:1}
 $M$ is locally Artinian over $\ll$.
\end{prop}

\begin{proof}
 We apply Proposition 7.6.1 in \cite{Dix} to the pair $(\ll, \ll\cap\pp)$. In particular, if $v \in M$, then $U(\ll) \cdot v$ has finite length as an $\ll-$module.
\end{proof}

\begin{cor}\label{corr:1}
 $M \in \mathcal{M}(\gg, Z_{U(\ll)})$.
\end{cor}

\begin{lem}\label{lemma:4}
  $\sigma(\ll,M) \subseteq \lbrace\theta_{(\eta|_{\hh_{\ll}})+\rho_{\ll}}^{\ll}\;|\;\eta \in \supp_{\hh}M\rbrace$. 
\end{lem}

\begin{proof}
The simple $\ll-$subquotients of $M$ are $(\ll,\ll\cap\pp)-$modules, and our claim follows the well-known relationship between the highest weight of a highest weight module and its central character.
\end{proof}
\vspace{0.5cm}

Let $N$ be a $\gg-$module, and let $\gg[N]$ be the set of elements $x \in \gg$ that act locally finitely in $N$. Then $\gg[N]$ is a Lie subalgebra of $\gg$, the \textit{Fernando-Kac subalgebra associated to} $N$. The fact has been proved independently by V. Kac in \cite{K} and by S. Fernando in \cite{F}.

\begin{theo}\label{th:1}
 Let $M_1$ be a non-zero subquotient of $M$. Assume that $\eta|_{\hh_{\ll}}$ is non-integral relative to $\ll$ for all $\eta \in \supp_{\hh}M$. Then $\ll \nsubseteq \gg[M_1]$.
\end{theo}

\begin{proof}
 By Lemma \ref{lemma:4}, no central character of $\ll$ in $M_1$ is $\ll-$integral. Therefore, no non-zero $\ll-$submodule of $M_1$ is finite dimensional. But $M_1 \neq 0$. Hence, $\ll \nsubseteq \gg[M_1]$.
\end{proof}

\vspace{0.5cm}

In agreement with \cite{PZ2}, we define a $\gg$-module $M$ to be a $(\gg,\kk)$-\textit{module} if $M$ is isomorphic as a $\kk$-module to a direct sum of isotypic components of $\kk$-types. If $M$ is a $(\gg,\kk)$-module, we write $M[\mu]$ for the $V(\mu)$-isotypic component of $M$, and we say that $V(\mu)$ is a $\kk$-\textit{type of} $M$ if $M[\mu]\neq 0$. We say that a $(\gg,\kk)$-module $M$ is \textit{of finite type} if $\dim M[\mu]\neq\infty$ for every $\kk$-type $V(\mu)$ of $M$. Sometimes, we also refer to $(\gg,\kk)$-modules of finite type as \textit{generalized Harish-Chandra modules}.

Note that for any $(\gg,\kk)-$module of finite type $M$ and any $\kk-$type $V(\sigma)$ of $M$, the finite-dimensional $\kk-$module $M[\sigma]$ is a $\tilde{\kk}-$module. In particular, $M$ is a $(\gg,\tilde{\kk})-$module of finite type. We will write $M \langle\delta\rangle$ for the $\tilde{\kk}-$isotypic components of $M$ where $\delta \in (\hh \cap \tilde{\kk})^*$.

If $M$ is a module of finite length, a $\kk$-type $V(\mu)$ of $M$ is \textit{minimal} if the function $\mu '\mapsto \parallel\mathrm{Re}\mu '+2\rho \parallel^2 $ defined on the set $\{\mu '\in\tt^*\; |\; M[\mu ']\neq 0\}$ has a minimum at $\mu$. Any non-zero $(\gg,\kk)$-module $M$ of finite length has a minimal $\kk$-type.
 
\subsection{Generalities on the Zuckerman functor}\label{sect:1.4}

Recall that the \textit{functor of $\kk$-finite vectors} $\Gamma_{\gg,\kk}^{\gg,\tt}$ is a well-defined left-exact functor on the category of $(\gg,\tt)$-modules with values in $(\gg,\kk)$-modules, $$\Gamma_{\gg,\kk}^{\gg,\tt}(M):=\sum_{M'\subset M, \dim M'=1, \dim U(\kk)\cdot M'<\infty}M'.$$ By $R^\cdot\Gamma_{\gg,\kk}^{\gg,\tt}:=\bigoplus_{i\geq 0}R^i\Gamma_{\gg,\kk}^{\gg,\tt}$ we denote as usual the total right derived functor of $\Gamma_{\gg,\kk}^{\gg,\tt}$, see \cite{Z} and the references therein.

\begin{prop}\label{prop:2}
If $\ll$ is any reductive subalgebra of $\gg$ containing $\kk$, then there is a natural isomorphism of $\ll-$modules 
\begin{equation}\label{iso_1}
R^{\cdot}\Gamma_{\gg,\kk}^{\gg,\tt}\,\,(N) \cong R^{\cdot}\Gamma_{\ll,\kk}^{\ll,\tt}\,\,(N).
\end{equation} 
\end{prop}

\begin{proof}
 See Proposition 2.5 in \cite{PZ4}.
\end{proof}

\begin{prop}\label{prop:3}
 If $\tilde{N} \in \mathcal{M}(\ll,\tt,Z_{U(\ll)}):=\mathcal{M}(\ll,Z_{U(\ll)}) \cap \mathcal{M}(\ll,\tt)$, then $$R^{\cdot}\Gamma_{\ll,\kk}^{\ll,\tt}\,\,(\tilde{N}) \in \mathcal{M}(\ll,\kk,Z_{U(\ll)}).$$

 Moreover, $$\sigma(\ll,R^{\cdot}\Gamma_{\ll,\kk}^{\ll,\tt}\,\,(\tilde{N})) \subset \sigma(\ll,\tilde{N}).$$
\end{prop}

\begin{proof}
 See Proposition 2.12 and Corollary 2.8 in \cite{Z}.
\end{proof}

\begin{cor}\label{corr:2}
 If $N \in \mathcal{M}(\gg, \tt, Z_{U(\ll)}):=\mathcal{M}(\gg,Z_{U(\ll)}) \cap \mathcal{M}(\gg,\tt)$, then $$R^{\cdot}\Gamma_{\gg,\kk}^{\gg,\tt}\,\,(N) \in \mathcal{M}(\gg,\kk,Z_{U(\ll)}).$$
 
 Moreover, $$\sigma(\ll,R^{\cdot}\Gamma_{\gg,\kk}^{\gg,\tt}\,\,(N)) \subseteq \sigma(\ll,N).$$
\end{cor}
\begin{proof}
 Apply Propositions \ref{prop:2} and \ref{prop:3}.
\end{proof}

Note that the isomorphism (\ref{iso_1}) enables us to write simply $\Gamma_{\kk,\tt}$ instead of $\Gamma_{\gg,\kk}^{\gg,\tt}$. 

For $\gg \supseteq \ll \supseteq \kk \supseteq \tt$ as above, let $\pp$ be a $\tt-$compatible parabolic subalgebra of $\gg$. It follows immediately that $\ll \cap \pp$ is a $\tt-$compatible parabolic subalgebra of $\ll$. Let $\hh_\ll \subset \ll \cap \pp$ be a Cartan subalgebra of $\ll$ containing $\tt$, and let $\hh \subset \pp$ be a Cartan subalgebra of $\gg$ such that $\hh_\ll = \hh \cap \ll$. We have the following diagram of subalgebras:

\begin{center}
\includegraphics[scale=0.55]{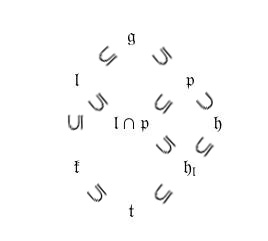} 
\end{center}

In this setup we have the following result.

\begin{theo}\label{th:2}
 Suppose $N \in \mathcal{M}(\gg,\pp) \cap \mathcal{M}(\gg,\tt)$, $M$ is a non-zero subquotient of $R^{\cdot}\Gamma_{\kk,\tt}\,\,(N)$ and $\eta|_{\hh_{\ll}}$ is not $\ll-$integral for all $\eta \in \supp_{\hh}N$. Then $\ll \nsubseteq \gg[M]$.
\end{theo}

\begin{proof}
 Every central character of $\ll$ in $M$ is a central character of $\ll$ in $N$. This follows from Corollary 2.8 in \cite{Z}. By our assumptions, no central character of $\ll$ in $N$ is $\ll-$integral. Hence, no $\ll-$submodule of $M$ is finite dimensional, and thus $\ll \nsubseteq \gg[M]$.
\end{proof}

\section{The fundamental series: main results}\label{sect:2}

We now introduce one of our main objects of study: the fundamental series of generalized Harish-Chandra modules.

We start by fixing some more notation: if $\qq$ is a subalgebra of $\gg$ and $J$ is a $\qq$-module, we set $\ind^{\gg}_\qq J:=U(\gg)\otimes_{U(\qq)}J$ and $\pro_\qq^\gg J:=\Hom_{U(\qq)}(U(\gg),J)$. For a finite-dimensional $\pp$- or $\bar{\pp}$-module $E$ we set $N_\pp(E):=\Gamma_{\tt,0}(\pro_\pp^\gg(E\otimes\Lambda^{\dim \nn}(\nn)))$, $N_{\bar{\pp}}(E^*):=\Gamma_{\tt,0}(\pro_{\bar{\pp}}^\gg(E^*\otimes\Lambda^{\dim \nn}(\nn^*)))$. One can show that both $N_\pp(E)$ and $N_{\bar{\pp}}(E^*)$ have simple socles as long as $E$ itself is simple.

The \textit{fundamental series} of $(\gg,\kk)$-modules of finite type $F^\cdot(\kk,\pp,E)$ is defined as follows. Let $\pp=\mm\crplus\nn$ be a minimal compatible parabolic subalgebra, $E$ be a simple finite dimensional $\pp$-module on which $\tt$ acts via the weight $\omega\in\tt^*$, and $\mu:=\omega+2\rho_\nn^\perp$ where $\rho_\nn^\perp :=\rho_\nn -\rho$. Set $$F^\cdot(\kk,\pp,E):=R^\cdot\Gamma_{\kk,\tt}(N_\pp(E)).$$ 

In the rest of the paper we assume that $\hh \cap \tilde{\kk}$ is a Cartan subalgebra of $\tilde{\kk}$.

\begin{theo}\label{th:3}
    \begin{enumerate}[a)]
		\item $F^\cdot(\kk,\pp,E)$ is a $(\gg,\kk)$-module of finite type and $Z_{U(\gg)}$ acts on $F^\cdot(\pp,E)$ via the $Z_{U(\gg)}$-character $\theta_{\nu+\tilde{\rho}}$ where $\tilde{\rho}:=\rho_{\ch_\hh\bb}$ for some Borel subalgebra $\bb$ of $\gg$ with $\bb\supset\hh,\;\bb\subset\pp$ and $\bb\cap\kk=\bb_\kk$, and where $\nu$ is the $\bb$-highest weight of $E$ (note that $\nu|_\tt=\omega$).
		\item $F^\cdot(\kk,\pp,E)$ is a $(\gg,\kk)$-module of finite length.
		\item There is a canonical isomorphism
		\begin{equation}\label{eq1}
		    F^\cdot(\kk,\pp,E)\simeq R^\cdot\Gamma_{\tilde{\kk},\tilde{\kk}\cap\mm}(\Gamma_{\tilde{\kk}\cap\mm,0}(\pro_\pp^\gg (E\otimes \Lambda^{\dim \nn}(\nn)))).
		\end{equation}
    \end{enumerate}
\end{theo}

\begin{proof}
	Part a) is a recollection of Theorem 2, a) in \cite{PZ2}. Part b) is a recollection of Theorem 2.5 in \cite{PZ5}. Part c) follows from the comparison principle (Proposition $2.6$) in \cite{PZ4}.
\end{proof}

\begin{cor}\label{corr:3}
    $F^\cdot(\kk,\pp,E)$ is a $(\gg,\tilde{\kk})$-module of finite type.
\end{cor}

\begin{proof}
    As we observed in subsection \ref{sect:1.3}, every $(\gg,\kk)-$module of finite type is a $(\gg,\tilde{\kk})-$module of finite type.
\end{proof}

\begin{cor}\label{corr:4}
    Let $\kk_1$ and $\kk_2$ be two algebraic reductive subalgebras such that $\tilde{\kk_1}=\tilde{\kk_2}$. Suppose that $\pp$ is a parabolic subalgebra which is both $\tt_1-$ and $\tt_2-$compatible and  $\tt_1-$ and $\tt_2-$minimal for some Cartan subalgebras $\tt_1$ of $\kk_1$ and $\tt_2$ of $\kk_2$. Then there exists a canonical isomorphism
	$$F^\cdot(\kk_1,\pp,E)\simeq F(\kk_2,\pp,E).$$
\end{cor}

\begin{proof}
    Consider the isomorphism (\ref{eq1}) for $\kk_1$ and $\kk_2$, and notice that $$R^\cdot\Gamma_{\tilde{\kk},\tilde{\kk}\cap\mm}(\Gamma_{\tilde{\kk}\cap\mm,0}(\pro_\pp^\gg (E\otimes \Lambda^{\dim \nn}(\nn))))$$ depends only on $\tilde{\kk}$ and $\pp$, but not on $\kk_1$ and $\kk_2$.
\end{proof}

\begin{cor}\label{corr:5}
 Let $M$ be any non-zero subquotient of $F^{\cdot}(\kk,\pp,E)$. If the $\bb-$highest weight $\nu \in \hh^*$ of $E$ is non-integral after restriction to $\hh \cap \ll$ for any reductive subalgebra $\ll$ of $\gg$ such that $\ll \supset \tilde{\kk}$, then $\tilde{\kk}$ is a maximal reductive subalgebra of $\gg[M]$.
\end{cor}

\begin{proof}
 Corollary \ref{corr:3} shows that $\tilde{\kk} \subseteq \gg[M]$. Theorem \ref{th:2} shows that if $\ll$ is a reductive subalgebra of $\gg$ such that $\ll$ is strictly larger than $\tilde{\kk}$, then $\ll \nsubseteq \gg[M]$. The assumption on $\nu$ implies that all weights in $\supp_{\hh \cap \ll}\,(N_{\pp}(E))$ are non-integral with respect to $\ll$.   
\end{proof}

\paragraph{\textbf{Example}} Here is an example to Corollary \ref{corr:5}. Let $\gg = F_4, \kk \simeq \so(3) \oplus \so(6)$. Then $\kk = \tilde{\kk}$. By inspection, there is only one proper intermediate subalgebra $\ll$, $\tilde{\kk} \subset \ll \subset \gg$, and $\ll$ is isomorphic to $\so(9)$. We have $\tt = \hh$, and $\varepsilon_1,\varepsilon_2,\varepsilon_3,\varepsilon_4$ is a standard basis of $\hh^*$, see \cite{Bou}. A weight $\nu = \sum_{i = 1}^{4}m_i\varepsilon_i$ is $\kk-$integral iff $m_1 \in \ZZ$ or $m_1 \in \ZZ + \frac{1}{2}$, and $(m_2,m_3,m_4) \in \ZZ^3$ or $(m_2,m_3,m_4) \in \ZZ^3+(\frac{1}{2},\frac{1}{2},\frac{1}{2})$. On the other hand, $\nu$ is $\ll-$integral if $(m_1,m_2,m_3,m_4) \in \ZZ^4$ or $(m_1,m_2,m_3,m_4) \in \ZZ^4 + (\frac{1}{2},\frac{1}{2},\frac{1}{2},\frac{1}{2})$. So if the $\bb-$highest weight $\nu$ of $E$ is not $\ll-$integral, Corollary \ref{corr:5} implies that $\gg[M] = \tilde{\kk}$ for any simple subquotient $M$ of $F^\cdot(\kk,\pp,E)$. 

\paragraph{\textbf{Remark}}
\begin{enumerate}[a)]

\item In \cite{PZ1} another method, based on the notion of a small subalgebra introduced by Willenbring and Zuckerman in \cite{WZ}, for computing maximal reductive subalgebras of simple subquotients of $F^\cdot(\kk,\pp,E)$ is suggested. Note that the subalgebra $\kk \simeq \so(3) \oplus \so(6)$ of $F_4$ considered in the above example is not small in $\so(9)$, so the above conclusion that $\gg[M] = \kk$ does not follow from \cite{PZ1}. On the other hand, if one replaces $\kk$ in the example by $\kk'\simeq\so(5)\oplus\so(4)$, then a conclusion similar to that of the example can be reached both by the method of \cite{PZ1} and by Corollary \ref{corr:5}.

\item There are pairs $(\gg,\kk)$ to which neither the method of \cite{PZ1} nor Corollary \ref{corr:5} apply. Such an example is a pair $(\gg = F_4, \kk\simeq\so(8))$. The only proper intermediate subalgebra in this case is $\ll \simeq \so(9)$; however $\so(8)$ is not small in $\so(9)$ and any $\kk=\tilde{\kk}-$integrable weight is also $\ll-$integrable.  
\end{enumerate}

If $M$ is a $(\gg,\kk)$-module of finite type, then $\Gamma_{\kk,0}(M^*)$ is a well-defined $(\gg,\kk)$-module of finite type and $\Gamma_{\kk,0}(\cdot^*)$ is an involution on the category of $(\gg,\kk)$-modules of finite type. We put $\Gamma_{\kk,0}(M^*):=M_\kk^*$. There is an obvious $\gg$-invariant non-degenerate pairing $M\times M_\kk^*\rightarrow\field{C}$.

\vspace{0.5cm}

The following five statements are recollections of the main results of \cite{PZ2} (Theorem 2 through Corollary 4).

\begin{theo}\label{th:4}
    Assume that $V(\mu)$ is a generic $\kk$-type and that $\pp=\pp_{\mu+2\rho}$ ($\mu$ is necessarily $\bb_\kk$-dominant and $\kk$-integral).
	\begin{enumerate}[a)]
	\item $F^i(\kk,\pp,E)=0$ for $i\neq s:=\dim\nn_\kk$ .
	\item There is a $\kk$-module isomorphism $$F^s(\kk,\pp,E)[\mu]\cong\field{C}^{\dim E}\otimes V(\mu),$$ and $V(\mu)$ is the unique minimal $\kk$-type of $F^s(\kk,\pp,E)$.
	\item Let $\bar{F}^s(\kk,\pp,E)$ be the $\gg$-submodule of $F^s(\kk,\pp,E)$ generated by $F^s(\kk,\pp,E)[\mu]$. Then $\bar{F}^s(\kk,\pp,E)$ is simple and $\bar{F}^s(\kk,\pp,E) = \Soc F^s(\kk,\pp,E)$. Moreover, $F^s(\kk,\pp,E)$ is cogenerated by $F^s(\kk,\pp,E)[\mu]$. This implies that $F^s(\kk,\pp,E)^*_\tt$ is generated by $F^s(\kk,\pp,E)^*_\tt[w_m(-\mu)]$, where $w_m\in W_\kk$  is the element of maximal length in the Weyl group $W_\kk$ of $\kk$.
	\item For any non-zero $\gg$-submodule $M$ of $F^s(\kk,\pp,E)$ there is an isomorphism of $\mm$-modules $$H^r(\nn,M)^\omega\cong E.$$
\end{enumerate}
\end{theo}

\begin{theo}\label{recth1}
 Let $M$ be a simple $(\gg,\kk)-$module of finite type with minimal $\kk-$type $V(\mu)$ which is generic . Then $\pp:=\pp_{\mu+2\rho} = \mm \crplus \nn$ is a minimal compatible parabolic subalgebra. Let $\omega:=\mu - 2\rho^{\perp}_n$ (recall that $\rho^{\perp}_\nn = \rho_{\mathrm{ch}_\tt(\nn \cap \kk^{\perp})}$), and let $E$ be the $\pp-$module $H^r(\nn,M)^{\omega}$ with trivial $\nn-$action, where $r = \mathrm{dim}(\nn\cap\kk^{\perp})$. Then $E$ is a simple $\pp-$module, the pair $(\pp,E)$ satisfies the hypotheses of Theorem \ref{th:4}, and $M$ is canonically isomorphic to $\bar{F}^s(\pp,E)$ for $s=\mathrm{dim}(\nn\cap\kk)$.
\end{theo}

\begin{cor}\label{corr1_pz2}
 (Generic version of a theorem of Harish-Chandra). There exist at most finitely many simple $(\gg,\kk)-$modules $M$ of finite type with a fixed $Z_{U(\gg)}-$character such that a minimal $\kk-$type of $M$ is generic. (Moreover, each such $M$ has a unique minimal $\kk-$type by Theorem \ref{th:4} b).)
\end{cor}

\begin{proof}
 By Theorems \ref{th:3} a) and \ref{recth1}, if $M$ is a simple $(\gg,\kk)-$module of finite type with generic minimal $\kk-$type $V(\mu)$ for some $\mu$, then the $Z_{U(\gg)}-$character of $M$ is $\theta_{\nu + \tilde{\rho}}$. There are finitely many Borel subalgebras $\bb$ as in Theorem \ref{th:3} a); thus, if $\theta_{\nu + \tilde{\rho}}$ is fixed, there are finitely many possibilities for the weight $\nu$ (as $\theta_{\nu + \tilde{\rho}}$ determines $\nu + \tilde{\rho}$ up to a finite choice). Therefore, up to isomorphism, there are finitely many possibilities for the $\pp-$module $E$, and hence, up to isomorphism, there are finitely many possibilities for $M$.
\end{proof}

\begin{theo}\label{recth2}
 Assume that the pair $(\gg,\kk)$ is regular, i.e. $\tt$ contains a regular element of $\gg$. Let $M$ be a simple $(\gg,\kk)-$module (a priori of infinite type) with a minimal $\kk-$type $V(\mu)$ which is generic. Then $M$ has finite type, and hence by Theorem \ref{recth1}, $M$ is canonically isomorphic to $\bar{F}^s(\pp,E)$ (where $\pp,E$ and $s$ are as in Theorem \ref{recth1}).
\end{theo}

\begin{cor}\label{corr2_pz2}
 Let the pair $(\gg,\kk)$ be regular.
 \begin{enumerate}[a)]
  \item There exist at most finitely many simple $(\gg,\kk)-$modules $M$ with a fixed $Z_{U(\gg)}-$character, such that a minimal $\kk-$type of $M$ is generic. All such $M$ are of finite type (and have a unique minimal $\kk-$type by Theorem \ref{th:4} b)).
  \item (Generic version of Harish-Chandra's admissibility theorem). Every simple $(\gg,\kk)-$module with a generic minimal $\kk-$type has finite type.
 \end{enumerate}
\end{cor}

\begin{proof}
 The proof of a) is as the proof of Corollary \ref{corr1_pz2} but uses Theorem \ref{recth2} instead of Theorem \ref{recth1}, and b) is a direct consequence of Theorem \ref{recth2}.
\end{proof}

\vspace{0.3cm}

The following statement follows from Corollary \ref{corr:5} and Theorem \ref{recth1}.

\begin{cor}\label{corr:6}
 Let $M$ be as in Theorem \ref{recth1}. If the $\bb-$highest weight of $E$ is not $\ll-$integral for any reductive subalgebra $\ll$ with $\tilde{\kk} \subset \ll \subseteq \gg$, then $\tilde{\kk}$ is a maximal reductive subalgebra of $\gg[M]$.
\end{cor}

\begin{defi}\label{def:1}
 Let $\pp \supset \bb_{\kk}$ be a minimal $\tt-$compatible parabolic subalgebra and let $E$ be a simple finite dimensional $\pp-$module on which $\tt$ acts by $\omega$. We say that the pair $(\pp,E)$ is \textit{allowable} if $\mu = \omega + 2\rho^{\perp}_{\nn}$ is dominant integral for $\kk$, $\pp_{\mu+2\rho} = \pp$, and $V(\mu)$ is generic.
\end{defi}

Theorem \ref{recth1} provides a classification of simple $(\gg,\kk)-$modules with generic minimal $\kk-$type in terms of allowable pairs. Note that for any minimal $\tt-$compatible parabolic subalgebra $\pp \supset \bb_{\kk}$, there exists a $\pp-$module $E$ such that $(\pp,E)$ is allowable.

\section{The case $\kk \simeq \sl(2)$}\label{sect:3}

Let $\kk \simeq \sl(2)$. In this case there is only one minimal $\tt-$compatible parabolic subalgebra $\pp = \mm \crplus \nn$ of $\gg$ which contains $\bb_{\kk}$. Furthermore, we can identify the elements of $\tt^*$ with complex numbers, and the $\bb_{\kk}-$dominant integral weights of $\tt$ in $\nn \cap \kk^{\perp}$ with non-negative integers. It is shown in \cite{PZ2} that in this case the genericity assumption on a $\kk-$type $V(\mu), \mu \geq 0,$ amounts to the condition $\mu \geq \Gamma:=\tilde{\rho}(h)-1$ where $h \in \hh$ is the semisimple element in a standard basis $e,h,f$ of $\kk \simeq \sl(2)$.

In our work \cite{PZ5} we have proved a different sufficient condition for the main results of \cite{PZ2} to hold when $\kk \simeq \sl(2)$. Let $\lambda_1$ and $\lambda_2$ be the maximum and submaximum weights of $\tt$ in $\nn \cap \kk^{\perp}$ (if $\lambda_1$ has multiplicity at least two in $\nn \cap \kk^{\perp}$, then $\lambda_2 = \lambda_1$; if $\dim \nn \cap \kk^{\perp} = 1$, then $\lambda_2 = 0$). Set $\Lambda := \frac{\lambda_1 + \lambda_2}{2}$.

\begin{theo}\label{th:5}
 If $\kk \simeq \sl(2)$, all statements of section \ref{sect:2} from Theorem \ref{th:4} through Corollary \ref{corr2_pz2} hold if we replace the assumption that $\mu$ is generic by the assumption $\mu \geq \Lambda$. As a consequence, the isomorphism classes of simple $(\gg,\kk)-$modules whose minimal $\kk-$type is $V(\mu)$ with $\mu \geq \Lambda$ are parameterized by the isomorphism classes of simple $\pp-$modules $E$ on which $\tt$ acts via $\mu - 2\rho_{\nn}^{\perp}$. 
\end{theo}

The $\sl(2)-$subalgebras of a simple Lie algebra are classified (up to conjugation) by Dynkin in \cite{D}. We will now illustrate the computation of the bound $\Lambda$ as well as the genericity condition on $\mu$ in examples.

We first consider three types of $\sl(2)-$subalgebras of a simple Lie algebra: long root$-\sl(2)$, short root$-\sl(2)$ and principal $\sl(2)$ (of course, there are short roots only for the series $B, C$ and for $G_2$ and $F_4$). We compare the bounds $\Lambda$ and $\Gamma$ in the following table.

{%
\newcommand{\mc}[3]{\multicolumn{#1}{#2}{#3}}
\begin{table}[ht]
\begin{center}
\begin{tabular}{|c|c|c|c|}\hline
× & long root & short root & principal\\\hline
\mc{1}{|l|}{$A_n, n\geq 2$} & \mc{1}{l|}{$\Gamma = n-1 \geq 1 = \Lambda$} & \mc{1}{l|}{not applicable} & \mc{1}{l|}{$\Gamma = \frac{n(n+1)(n+2)}{6}-1 \geq 2n-1 = \Lambda$}\\\hline
\mc{1}{|l|}{$B_n, n\geq 2$} & \mc{1}{l|}{$\Gamma = 2n-3 \geq 1 = \Lambda$} & \mc{1}{l|}{$\Gamma = 2n-2 \geq 2 = \Lambda$} & \mc{1}{l|}{$\Gamma = \frac{n(n+1)(4n-1)}{6}-1 > 4n-3 = \Lambda$}\\\hline
\mc{1}{|l|}{$C_n, n\geq 3$} & \mc{1}{l|}{$\Gamma = n-1 > 1 = \Lambda$} & \mc{1}{l|}{$\Gamma = 2n-2 > 2 = \Lambda$} & \mc{1}{l|}{$\Gamma = \frac{n(n+1)(2n+1)}{3}-1 > 4n-3 = \Lambda$}\\\hline
\mc{1}{|l|}{$D_n, n\geq 4$} & \mc{1}{l|}{$\Gamma = 2n-4 > 1 = \Lambda$} & \mc{1}{l|}{not applicable} & \mc{1}{l|}{$\Gamma = \frac{2(n-1)n(n+1)}{3}-1 > 4n-7 = \Lambda$}\\\hline
\mc{1}{|l|}{$E_6$} & \mc{1}{l|}{$\Gamma = 10 > 1 = \Lambda$} & \mc{1}{l|}{not applicable} & \mc{1}{l|}{$\Gamma = 155 > 21 = \Lambda$}\\\hline
\mc{1}{|l|}{$E_7$} & \mc{1}{l|}{$\Gamma = 16 > 1 = \Lambda$} & \mc{1}{l|}{not applicable} & \mc{1}{l|}{$\Gamma = 398 > 33 = \Lambda$}\\\hline
\mc{1}{|l|}{$E_8$} & \mc{1}{l|}{$\Gamma = 28 > 1 = \Lambda$} & \mc{1}{l|}{not applicable} & \mc{1}{l|}{$\Gamma = 1239 > 57 = \Lambda$}\\\hline
\mc{1}{|l|}{$F_4$} & \mc{1}{l|}{$\Gamma = 7 > 1 = \Lambda$} & \mc{1}{l|}{$\Gamma = 10 > 2 = \Lambda$} & \mc{1}{l|}{$\Gamma = 109 > 21 = \Lambda$}\\\hline
\mc{1}{|l|}{$G_2$} & \mc{1}{l|}{$\Gamma = 2 > 1 = \Lambda$} & \mc{1}{l|}{$\Gamma = 4 > 3 = \Lambda$} & \mc{1}{l|}{$\Gamma = 15 > 9 = \Lambda$}\\\hline
\end{tabular}
\end{center}
\end{table}
}%

\begin{center}
 \textbf{Table A}
\end{center}

\vspace{0.5cm}	

Let's discuss the case $\gg = F_4$ in more detail. Recall that the \textit{Dynkin index} of a semisimple subalgebra $\ss \subset \gg$ is the quotient of the normalized $\gg-$invariant summetic bilinear form on $\gg$ restricted to $\ss$ and the normalized $\ss-$invariant symmetric bilinear form on $\ss$, where for both $\gg$ and $\ss$ the square length of a long root equals $2$. According to Dynkin \cite{D}, the conjugacy class of an $\sl(2)-$subalgebra $\kk$ of $F_4$ is determined by the Dynkin index of $\kk$ in $F_4$. Moreover, for $\gg = F_4$ the following integers are Dynkin indices of $\sl(2)-$subalgebras: $1$(long root), $2$(short root), $3,4,6,8,9,10,11,12,28,35,36,60,156.$ The bounds $\Lambda$ and $\Gamma$ are given in the following table.

\begin{table}[ht]
\begin{center}
\begin{tabular}{|c|c|c|c|}\hline
Dynkin index & 1 & 2 & 3\\\hline 
  & $\Gamma = 7 > 1 =\Lambda$ & $\Gamma = 10 > 2 =\Lambda$ & $\Gamma = 14 > 3 =\Lambda$\\\hline
Dynkin index & 4 & 6 & 8\\\hline
  & $\Gamma = 15 > 3 =\Lambda$ & $\Gamma = 16 > 4 =\Lambda$ & $\Gamma = 17 > 4 =\Lambda$\\\hline
Dynkin index & 9 & 10 & 11\\\hline
  & $\Gamma = 25 > 5 =\Lambda$ & $\Gamma = 26 > 5 =\Lambda$ & $\Gamma = 28 > 6 =\Lambda$\\\hline
Dynkin index & 12 & 28 & 35\\\hline
  & $\Gamma = 29 > 6 =\Lambda$ & $\Gamma = 45 > 9 =\Lambda$ & $\Gamma = 50 > 10 =\Lambda$\\\hline
Dynkin index & 36 & 60 & 156\\\hline
  & $\Gamma = 51 > 10 =\Lambda$ & $\Gamma = 67 > 13 =\Lambda$ & $\Gamma = 109 > 21 =\Lambda$\\\hline
\end{tabular}
\end{center}
\end{table}

\begin{center}
 \textbf{Table B}
\end{center}

\vspace{0.5cm}

We conclude this section by recalling a conjecture from \cite{PZ5}. Let $\mathcal{C}_{\bar{\pp},\tt,n}$ denote the full subcategory of $\gg-$mod consisting of finite-length modules with simple subquotients which are $\bar{\pp}-$locally finite $(\gg,\tt)-$modules $N$ whose $\tt-$weight spaces $N^{\beta},\;\beta \in \mathbb{Z}$, satisfy $\beta \geq n$. Let $\mathcal{C}_{\kk,n}$ be the full subcategory of $\gg-$mod consisting of finite length modules whose simple subquotients are $(\gg,\kk)-$modules with minimal $\kk \simeq \sl(2)-$type $V(\mu)$ for $\mu \geq n$. We show in \cite{PZ5} that the functor $R^{1}\Gamma_{\kk,\tt}$ is a well-defined fully faithful functor from $\mathcal{C}_{\pp,\tt,n+2}$ to $\mathcal{C}_{\kk,n}$ for $n \geq 0$. Moreover, we make the following conjecture.

\begin{conject}\label{conj:1} 
 Let $n \geq \Lambda$. Then $R^{1}\Gamma_{\kk,\tt}$ is an equivalence between the categories $\mathcal{C}_{\bar{\pp},\tt,n+2}$ and $\mathcal{C}_{\kk,n}$.
\end{conject}

 We have proof of this conjecture for $\gg \simeq \sl(2)$ and, jointly with V. Serganova, for $\gg \simeq \sl(3)$.

\section{Eligible subalgebras}\label{sect:4}

In what follows we adopt the following terminology. A \textit{root subalgebra} of $\gg$ is a subalgebra which contains a Cartan subalgebra of $\gg$. An $r$-\textit{subalgebra} of $\gg$ is a subalgebra $\ll$ whose root spaces (with respect to a Cartan subalgebra of $\ll$) are root spaces of $\gg$. The notion of $r$-subalgebra goes back to \cite{D}. A root subalgebra is, of course, an $r$-subalgebra.

We now give the following key definition.

\begin{defi}\label{def:2}
    An algebraic reductive in $\gg$ subalgebra $\kk$ is eligible if $C(\tt)=\tt+C(\kk)$.
\end{defi}

Note that in the above definition one can replace $\tt$ with any Cartan subalgebra of $\kk$. Furthermore, if $\kk$ is eligible then $\hh\subset C(\tt)=\tt+C(\kk)\subset\tilde{\kk}=\kk+C(\kk)$, i.e. $\hh$ is a Cartan subalgebra of both $\tilde{\kk}$ and $\gg$. In particular, $\tilde{\kk}$ is a reductive root subalgebra of $\gg$. As $\kk$ is an ideal in $\tilde{\kk}$, $\kk$ is an $r$-subalgebra of $\gg$.

\begin{prop}\label{prop:4}
 Assume $\kk$ is an $r-$subalgebra of $\gg$. The following three conditions are equivalent:
\begin{itemize}
  \item[(i)] $\kk\;is\;eligible$;
  \item[(ii)] $\,C(\kk)_{ss} = C(\tt)_{ss}$; 
  \item[(iii)] $\;\,\dim\,C(\kk)_{ss} = \dim\,C(\tt)_{ss}$.
\end{itemize}
\end{prop}

\begin{proof}
 The implications (i)$\Rightarrow$(ii)$\Rightarrow$(iii) are obvious. To see that (iii) implies (i), observe that if $\kk$ is an $r-$subalgebra of $\gg$, then $\hh \subseteq \tt + C(\kk) \subseteq C(\tt)$. Therefore the inclusion $\tt + C(\kk) \subseteq C(\tt)$ is proper if and only if $\gg^{\pm\alpha}\in C(\tt)\backslash C(\kk)$ for some root $\alpha \in \Delta$, or, equivalently, if the inclusion $C(\kk)_{ss} \subseteq C(\tt)_{ss}$ is proper.
\end{proof}

\vspace{0.3cm}

An algebraic, reductive in $\gg$, $r$-subalgebra $\kk$ may or may not be eligible. If $\kk$ is a root subalgebra, then $\kk$ is always eligible. If $\gg$ is simple of types $A,C,D$ and $\kk$ is a semisimple $r$-subalgebra, then $\kk$ is necessarily eligible. In general, a semisimple $r$-subalgebra is eligible if and only if the roots of $\gg$ which vanish on $\tt$ are strongly orthogonal to the roots of $\kk$. For example, if $\gg$ is simple of type $B$ and $\kk$ is a simple $r$-subalgebra of type $B$ of rank less or equal than $\rk\gg-2$, then $C(\kk)_{ss}$ is simple of type $D$ whereas $C(\tt)_{ss}$ is simple of type $B$. Hence in this case $\kk$ is not eligible.

Note, however that any semisimple $r$-subalgebra $\kk '$ can be extended to an eligible subalgebra $\kk$ just by setting $\kk:=\kk '+\hh_{C(\kk ')}$ where $\hh_{C(\kk ')}$ is a Cartan subalgebra of $C(\kk ')$. Finally, note that if $x$ is any algebraic regular semisimple element of $C(\kk ')$, then $\kk:= \kk'\oplus Z(C(\kk')) +\mathbb{C}x$ is an eligible subalgebra of $\gg$. Indeed, if $\tt '\subseteq\kk '$ is a Cartan subalgebra of $\kk '$, and $\hh_{\kk} := \tt' \oplus Z(C(\kk')) + \mathbb{C}x$ is the corresponding Cartan subalgebra of $\kk$, then $C(\hh_{\kk})$ is a Cartan subalgebra of $\gg$. Hence, 

\begin{equation}\label{eq2}
    C(\hh_{\kk})=\hh_{\kk}+C(\kk)
\end{equation}
as the right-hand side of (\ref{eq2}) necessarily contains a Cartan subalgebra of $\gg$.

To any eligible subalgebra $\kk$ we assign a unique weight $\varkappa\in\hh^*$ (the ``canonical weight associated with $\kk$''). It is defined by the conditions $\varkappa|_{(\hh\cap\kk_{ss})}=\rho,\; \varkappa|_{(\hh\cap C(\kk))}=0$.

\section{The generalized discrete series}\label{sect:5}

In what follows we assume that $\kk$ is eligible and $\hh \subset \tilde{\kk}$. In this case $\hh$ is a Cartan subalgebra both of $\tilde{\kk}$ and $\gg$. Let $\lambda\in\hh^*$ and set $\gamma := \lambda|_{\tt}$. Assume that $\mm:=\mm_{\gamma}=C(\tt)$. Assume furthermore that $\lambda$ is $\mm$-integral and let $E_\lambda$ be a simple finite-dimensional $\mm$-module with $\bb-$highest weight $\lambda$. Then $$D(\kk,\lambda):=F^s(\kk,\pp_{\gamma},E_\lambda\otimes\Lambda^{\dim\nn_{\gamma}}(\nn_{\gamma}^*))$$
is by definition a \textit{generalized discrete series module}.

Note that since $D(\kk,\lambda)$ is a fundamental series module, Theorem \ref{th:3} applies to $D(\kk,\lambda)$. In the case when $\kk$ is a root subalgebra and $\lambda$ is regular, we have $\lambda = \gamma$ and $\pp_{\gamma}$ is a Borel subalgebra of $\gg$ which we denote by $\bb_{\lambda}$. Then $D(\kk,\lambda) = R^s\Gamma_{\kk,\hh}(\Gamma_\hh(\pro^\gg_{\bb_\lambda}\,E_\lambda))$, i.e. $D(\kk,\lambda)$ is cohomologically co-induced from a $1-$dimensional $\bb_{\lambda}-$module. If in addition, $\kk$ is a symmetric subalgebra, $\lambda$ is $\kk$-integral, and $\lambda - \tilde{\rho}$ is $\bb_{\lambda}-$dominant regular, then $D(\kk,\lambda)$ is a $(\gg,\kk)$-module in Harish-Chandra's discrete series, see \cite{KV}, Ch.XI.

Suppose $\kk$ is eligible but $\kk$ is not a root subalgebra. Suppose further that $\tilde{\kk}$ is symmetric. Any simple subquotient $M$ of $D(\kk,\lambda)$ is a $(\gg,\tilde{\kk})-$module and thus a Harish-Chandra module for $(\gg,\tilde{\kk})$. However, $M$ may or may not be in the discrete series of $(\gg,\tilde{\kk})-$modules. This becomes clear in Theorem \ref{th:6} below.

Our first result is a sharper version of the main result of \cite{PZ3} for an eligible $\kk$. 

\begin{theo}\label{th3}
    Let $\kk\subseteq\gg$ be eligible. Assume that $\lambda-2\varkappa$ is $\tilde{\kk}$-integral and dominant. Then, $D(\kk,\lambda) \neq 0$. Moreover, if we set $\mu := (\lambda - 2\varkappa)|_{\tt}$, then $V(\mu)$ is the unique minimal $\kk-$type of $D(\kk,\lambda)$. Finally, there are isomorphisms of simple finite-dimensional $\tilde{\kk}-$modules
$$D(\kk,\lambda)[\mu] \cong D(\kk,\lambda)\langle\lambda - 2\varkappa\rangle \simeq V_{\tilde{\kk}}(\lambda - 2\varkappa).$$

\end{theo}

\begin{proof}
 Note that $\mu = \gamma -2\rho.$ By Lemma 2 in \cite{PZ3} $$\mathrm{dim}\,\Hom_\kk(V(\mu),D(\kk,\lambda))=\mathrm{dim}\,E_{\lambda},$$ and hence $D(\kk,\lambda) \neq 0$. In addition, $V(\mu)$ is the unique minimal $\kk-$type of $D(\kk,\lambda)$. By construction, $D(\kk,\lambda)[\mu]$ is a finite-dimensional $\tilde{\kk}-$module. We will use Theorem \ref{th:3} c) to compute $D(\kk,\lambda)[\mu]$ as a $\tilde{\kk}-$module. Since $\kk$ is eligible, we have $\mm = \tt+C(\kk)$. As $[\tt,C(\kk)] = 0$ and $\tt$ is toral, the restriction of $E_{\lambda}$ to $C(\kk)$ is simple. We have
$$\tilde{\kk} = \kk_{ss} \oplus C(\kk),$$
and hence there is an isomorphism of $\tilde{\kk}-$modules
$$V_{\tilde{\kk}}(\lambda-2\varkappa) \cong (V(\mu)|_{\kk_{ss}})\boxtimes E_{\lambda}.$$
Consequently, we have isomorphisms of $C(\kk)-$modules
\begin{equation}\label{eq_th2}
 \Hom_{\kk} (V(\mu),V_{\tilde{\kk}}(\lambda - 2\varkappa)) \cong \Hom_{\kk_{ss}} ((V(\mu)|_{\kk_{ss}}),V_{\tilde{\kk}}(\lambda - 2\varkappa)) \cong E_{\lambda}.
\end{equation}
 
Write $\pp_\gamma = \pp$ and note that $\tilde{\kk} \cap \mm = \mm$. By Theorem \ref{th:3} c), we have a canonical isomorphism
$$D(\kk,\lambda) \cong R^s\Gamma_{\tilde{\kk},\mm}(\Gamma_{\mm,0}(\pro^{\gg}_{\pp}E_\lambda)).$$
According to the theory of the bottom layer \cite{KV}, Ch.V, Sec.6, $D(\kk,\lambda)$ contains the $\tilde{\kk}-$module
$$R^s\Gamma_{\tilde{\kk},\mm}(\Gamma_{\mm,0}(\mathrm{pro}_{\tilde{\kk}\cap\pp}^{\tilde{\kk}}E_\lambda))$$
which is in turn isomorphic to $V_{\tilde{\kk}}(\lambda-2\varkappa)$.

By the above argument, we have a sequence of injections
$$V_{\tilde{\kk}}(\lambda-2\varkappa) \hookrightarrow D(\kk,\lambda)\langle\lambda-2\varkappa\rangle \hookrightarrow D(\kk,\lambda)[\mu].$$

We conclude from (\ref{eq_th2}) that the above sequence of injections is in fact a sequence of isomorphisms of simple $\tilde{\kk}-$modules.
\end{proof}

\begin{cor}\label{corr3}
 Under the assumptions of Theorem \ref{th3}, there exists a simple $(\gg,\kk)-$module $M$ of finite type over $\kk$, such that if $V(\mu_M)$ is a minimal $\kk-$type of $M$, then $V(\mu_M)$ is the unique minimal $\kk-$type of $M$ and there is an isomorphism of finite-dimensional $\tilde{\kk}-$modules
$$M[\mu_M] \cong V_{\tilde{\kk}}(\lambda - 2\varkappa).$$ 
In particular, $M[\mu_M]$ is a simple $\tilde{\kk}-$submodule of $M$.
\end{cor}

\begin{proof}
First we construct a module $M$ as required. Let $\bar{D}(\kk,\lambda)$ be the $U(\gg)-$submodule of $D(\kk,\lambda)$ generated by the $\tilde{\kk}-$isotypic component $D(\kk,\lambda)\langle\lambda-2\varkappa\rangle$. Suppose $N$ is a proper $\gg-$submodule of $\bar{D}(\kk,\lambda)$. Since $D(\kk,\lambda)\langle\lambda-2\varkappa\rangle$ is simple over $\tilde{\kk}$,
$$N \cap (D(\kk,\lambda)\langle\lambda-2\varkappa\rangle) = 0.$$

Thus, if $N(\kk,\lambda)$ is the maximum proper submodule of $\bar{D}(\kk,\lambda)$, the quotient module 
$$M = \bar{D}(\kk,\lambda)/N(\kk,\lambda)$$
is a simple $(\gg,\tilde{\kk})-$module, and $M$ has finite type over $\kk$. Writing $\mu_M = \mu = \gamma-2\rho$, we see that $M$ has unique minimal $\kk-$type $V(\mu_M)$. Finally, by Theorem \ref{th3}, we have an isomorphism of finite-dimensional $\tilde{\kk}-$modules,
$$M[\mu_M] \cong V_{\tilde{\kk}}(\lambda-2\varkappa).$$ 
\end{proof}

If $\kk$ is symmetric (and hence $\kk$ is a root subalgebra due to the eligibility of $\kk$), Theorem \ref{th3} and Corollary \ref{corr3} go back to \cite{V} (where they are proven by a different method).

The following two statements are consequences of the main results of section \ref{sect:2} and Theorem \ref{th3}.

\begin{cor}\label{corr:8}
    Let $\kk$ be eligible, $\lambda\in\hh^*$ be such that $\lambda-2\varkappa$ is $\tilde{\kk}$-integral and $V(\mu)$ is generic for $\mu:= \lambda|_{\tt}-2\rho$.
	\begin{enumerate}[a)]
	    \item $\Soc\,D(\kk,\lambda)$ is a simple $(\gg,\kk)$-module with unique minimal $\kk$-type $V(\mu)$.
		\item There is a canonical isomorphism of $C(\kk)$-modules
		$$\Hom_\kk(V(\mu),\Soc\,D(\kk,\lambda))\simeq E_\lambda.$$
		\item There is a canonical isomorphism of $\tilde{\kk}$-modules
		$$V(\mu)\otimes\Hom_\kk(V(\mu),\Soc\,D(\kk,\lambda))\simeq V_{\tilde{\kk}}(\lambda-2\varkappa),$$
		i.e. the $V(\mu)$-isotypic component of $\Soc D(\kk,\lambda)$ is a simple $\tilde{\kk}$-module isomorphic to $V_{\tilde{\kk}}(\lambda-2\varkappa)$.
		\item If $\lambda-2\varkappa$ is not $\ll-$integral for any reductive subalgebra $\ll$ such that $\tilde{\kk} \subset \ll \subseteq \gg$, then $\tilde{\kk}$ is a maximal reductive subalgebra of $\gg[M]$ for any subquotient $M$ of $D(\kk,\lambda)$, in particular of $\Soc\,D(\kk,\lambda)$.
	\end{enumerate}
\end{cor}

\begin{proof}

    \textit{a)} Observe that $\pp_{\gamma} = \pp_{\mu+2\rho}$, and $D(\kk,\lambda)=F^s(\kk,\pp_{\mu+2\rho},\,E_{\lambda} \otimes \Lambda^{\mathrm{dim}\,\nn}(\nn^*))$. So, a) follows from Theorem \ref{th:4} c).

\vspace{0.3cm}

 \textit{b)} By Theorem \ref{th:4} c), $\Hom_{\kk}(V(\mu),\;\Soc\,D(\kk,\lambda)) = \Hom_{\kk}(V(\mu),\;D(\kk,\lambda))$, which in turn is isomorphic to $\Hom_{\kk}(V(\mu),V_{\tilde{\kk}}(\lambda-2\varkappa))$ by Theorem \ref{th3}. The desired isomorphism follows now from (\ref{eq_th2}).

\vspace{0.3cm}

 \textit{c)} This follows from the isomorphism in b) and the isomorphism $V(\mu) \otimes E_{\lambda} \cong V_{\tilde{\kk}}(\lambda-2\varkappa)$ of $\tilde{\kk}-$modules.
 
 \vspace{0.3cm}
 
 \textit{d)} Follows from Corollary \ref{corr:5}. Note that, since $\kk$ is eligible, $\tilde{\kk}$ is a root subalgebra and the condition that $\lambda - 2\varkappa$ be not $\ll-$integral involves only finitely many subalgebras $\ll$. 
\end{proof}

\begin{cor}\label{corr:9}
 Let $\kk$ be eligible and let $V(\mu)$ be a generic $\kk-$type.
      \begin{enumerate}[a)]
       \item Let $M$ be a simple $(\gg,\kk)-$module of finite type with minimal $\kk-$type $V(\mu)$. Then $M[\mu]$ is a simple finite-dimensional $\tilde{\kk}-$module isomorphic to $V_{\tilde{\kk}}(\lambda)$ for some weight $\lambda \in \hh^*$ such that $\lambda|_{\tt} = \mu + 2\rho$ and $\mu - 2\varkappa$ is $\tilde{\kk}-$integral. Moreover,$$M \cong \Soc\,D(\kk,\lambda).$$
       If in addition $\lambda$ is not $\ll-$integral for any reductive subalgebra $\ll$ with $\tilde{\kk} \subset \ll \subseteq \gg$, then $\tilde{\kk}$ is a unique maximal reductive subalgebra of $\gg[M]$.
       
	\item If $\kk$ is regular in $\gg$, then a) holds for any simple $(\gg,\kk)-$module with generic minimal $\kk-$type $V(\mu)$. In particular $M$ has finite type over $\kk$.
      \end{enumerate}
\end{cor}

\begin{proof}

 \textit{a)} We apply Theorem \ref{recth1}. Since $V(\mu)$ is generic, $\pp=\pp_{\mu+2\rho} = \mm \crplus \nn$ is a minimal $\tt-$compatible parabolic subalgebra. Let $\omega :=\mu - 2\rho_{\nn}^{\perp}$ (recall that $\rho_{\nn}^{\perp} = \rho_{\nn} - \rho $) and let $Q$ be the $\mm-$module $H^r(\nn,M)^{\omega}$ where $r=\mathrm{dim}(\kk^{\perp} \cap \nn)$.

Observe that $Q$ is a simple $\mm-$module and $M$ is canonically isomorphic to $\bar{F}^s(\pp,Q) = \Soc\,F^s(\pp,Q)$. Let $\lambda \in \hh^*$ be so that $\lambda - 2\tilde{\rho}_{\nn}$ is an extreme weight of $\hh$ in $Q$. Thus, $F^s(\pp,Q) = F^s(\pp,\,E_{\lambda} \otimes \Lambda ^{\mathrm{dim}\,\nn}(\nn^*)) = D(\kk,\lambda)$. Finally, $M \cong \Soc\,D(\kk,\lambda)$, and $\lambda|_{\tt} = \mu + 2\rho$. It follows that $\lambda - 2\varkappa$ is both $\kk-$integral and $C(\kk)-$integral. Hence, the weight $\lambda - 2\varkappa$ is $\tilde{\kk}-$integral.

\vspace{0.3cm}

 \textit{b)} We apply Theorem \ref{recth2}.
\end{proof}

\begin{cor}\label{corr:7}
    If $\kk\simeq \sl(2)$, the genericity assumption on $V(\mu)$ in Corollaries \ref{corr:8} and \ref{corr:9} can be replaced by the assumption $\mu\geq\Lambda$.
\end{cor}

\begin{proof}
The statement follows directly from Theorem \ref{th:5}.
\end{proof}

\vspace{0.3cm}

We conclude this paper by discussing in more detail an example of an eligible $\sl(2)-$subalgebra. Note first that if $\gg$ is any simple Lie algebra and $\kk$ is a long root $\sl(2)-$subalgebra, then the pair $(\gg,\tilde{\kk})$ is a symmetric pair. This is a well-known fact and it implies in particular that any $(\gg,\kk)-$module of finite type and of finite length is a Harish-Chandra module for the pair $(\gg,\tilde{\kk})$. The latter modules are classified under the assumption of simplicity see \cite{KV}, Ch.XI; however, in general, it is an open problem to determine which simple $(\gg,\tilde{\kk})-$modules have finite type over $\kk$. Without having been explicitly stated, this problem has been discussed in the literature, see \cite{OW} and the references therein. On the other hand, in this case $\Lambda = 1$, hence Corollaries \ref{corr:9} and \ref{corr:7} provide a classification of simple $(\gg,\kk)-$modules of finite type with minimal $\kk-$types $V(\mu)$ for $\mu \geq 1$. So the above problem 
reduces to matching the 
above two classifications in the case $\mu \geq 1$, and finding all simple $(\gg,\kk)-$modules of finite type whose minimal $\kk-$type equals $V(0)$ among the simple Harish-Chandra modules for the pair $(\gg,\tilde{\kk})$. We do this here in a special case.

Let $\gg = \sp(2n+2)$ for $n \geq 2$. By assumption, $\kk$ is a long root $\sl(2)-$subalgebra, and $\tilde{\kk} \simeq \sp(2n)\oplus\kk$. Consider simple $(\gg,\tilde{\kk})-$modules with $Z_{U(\gg)}-$character equal to the character of a trivial module. According to the Langlards classification, there are precisely $(n+1)^2$ pairwise non-isomorphic such modules, one of which is the trivial module. Following \cite{Co} (see figure 4.5 on page 93) we enumerate them as $\sigma_{t}$ for $0 \leq t \leq n$ and $\sigma_{ij}$ for $0 \leq i \leq n-1, 1 \leq j \leq 2n, i < j, i+j \leq 2n$. The modules $\sigma_t$ are discrete series modules. The modules $\sigma_{ij}$ are Langlands quotients of the principal series (all of them are proper quotients in this case).

We announce the following result which we intend to prove elsewhere.

\begin{theo}\label{th:6}
 Let $\gg = \sp(2n+2)$ for $n \geq 2$ and $\kk$ be a long root $\sl(2)-$subalgebra.
 
 a) Any simple $(\gg,\kk)-$module of finite type is isomorphic to a subquotient of the generalized discrete series module $D(\kk,\lambda)$ for some $\tilde{\kk} = \sp (2n)\oplus\kk-$integral weight $\lambda - 2\varkappa$.
 
 b) The modules $\sigma_0, \sigma_{0i}$ for $i = 1, \dots,2n,\sigma_{12}$ are, up to isomorphism, all of the simple $(\gg,\kk)-$modules of finite type whose $Z_{U(\gg)}$-character equals that of a trivial $\gg-$module. Moreover, their minimal $\kk-$types are as follows:
 \begin{center}
\begin{tabular}{|c|c|}\hline
\textbf{module} & \textbf{minimal $\kk-$type}\\\hline
$\sigma_0$ & $V(2n)$\\\hline
$\sigma_{0j}, n+1 \leq j \leq 2n$ & $V(j-1)$\\\hline
$\sigma_{0j}, 2 \leq j \leq n$ & $V(j-2)$\\\hline
$\sigma_{01}$ (trivial representation) & $V(0)$\\\hline
$\sigma_{12}$ & $V(0)$\\\hline
\end{tabular}\hspace{0.2cm}.
\end{center}
\end{theo}


\newpage

\begin{thebibliography}{20}
\bibitem[Bou] {Bou} N. Bourbaki, Groupes et Alg\`{e}bres de Lie, Ch.VI, Hermann, Paris, 1975.

\vspace{.25cm}

\bibitem[Co]{Co} D. H. Collingwood, Representations of rank one Lie groups, Pitman, Boston, 1985.

\vspace{.25cm}

\bibitem[D] {D} E. Dynkin, Semisimple subalgebras of semisimple Lie algebras, Mat. Sbornik (N.S.) \textbf{30} (72) (1952), 349 - 462 (Russian); English: Amer. Math. Soc. Transl. \textbf{6} (1957), 111 - 244.

\vspace{.25cm}

\bibitem[Dix]{Dix} J. Dixmier, Enveloping algebras, North Holland, Amsterdam, 1977.

\vspace{.25cm}

\bibitem[F]{F} S. Fernando, Lie algebra modules with finite-dimensional weight spaces I, Transactions AMS Soc \textbf{332} (1990), 757-781.

\vspace{.25cm}

\bibitem[K]{K} V. Kac, Constructing groups associated to infinite-dimensional Lie algebras, in:Infinite-dimensional groups with applications (Berkeley, 1984), Math. Sci. Res. Inst. Publ. 4, pp. 167-216.

\vspace{.25cm}

\bibitem[KV] {KV} A. Knapp, D. Vogan, Cohomological Induction and Unitary Representations, Princeton Mathematical Series, Princeton University Press, 1995.

\vspace{.25cm}

\bibitem[OW] {OW} B. Orsted, J.A. Wolf, Geometry of the Borel de Siebenthal discrete series, Journal of Lie Theory \textbf{20} (2010), 175-212.

\vspace{.25cm}

\bibitem[Pe]{Pe} A. V. Petukhov, Bounded reductive subalgebras of $\sl(n)$, Transformation Groups \textbf{16} (2011), 1173-1182. 

\vspace{.25cm}

\bibitem[PS1]{PS1} I. Penkov, V. Serganova, Generalized Harish-Chandra modules, Moscow Math. Journal \textbf{2} (2002), 753-767.

\vspace{.25cm}

\bibitem[PS2]{PS2} I. Penkov, V. Serganova, Bounded simple $(\gg,\sl(2))$-modules for $\rk\gg = 2$, Journal of Lie Theory \textbf{20} (2010), 581-615.

\vspace{.25cm}

\bibitem[PS3]{PS3} I. Penkov, V. Serganova, On bounded generalized Harish-Chandra modules, Annales de l'Institut Fourier \textbf{62} (2012), 477-496.

\vspace{.25cm}

\bibitem[PSZ]{PSZ} I. Penkov, V. Serganova, G. Zuckerman, On the existence of $(\gg,\kk)-$modules of finite type, Duke Math. Journal \textbf{125} (2004), 329-349.

\vspace{.25cm}

\bibitem[PZ1] {PZ1} I. Penkov, G. Zuckerman, Generalized Harish-Chandra modules: a new direction of the structure theory of representations, Acta Applicandae Mathematicae \textbf{81}(2004), 311 - 326.

\vspace{.25cm}

\bibitem[PZ2] {PZ2} I. Penkov, G. Zuckerman, Generalized Harish-Chandra modules with generic minimal $\kk$-type, Asian Journal of Mathematics \textbf{8} (2004), 795 - 812.

\vspace{.25cm}

\bibitem[PZ3] {PZ3} I. Penkov, G. Zuckerman, A construction of generalized Harish-Chandra modules with arbitrary minimal $\kk-$type, Canad. Math. Bull. \textbf{50} (2007), 603-609.

\vspace{.25cm}

\bibitem[PZ4] {PZ4} I. Penkov, G. Zuckerman, A construction of generalized Harish-Chandra modules for locally reductive Lie algebras, Transformation Groups \textbf{13} (2008), 799-817.

\vspace{.25cm}

\bibitem[PZ5] {PZ5} I. Penkov, G. Zuckerman, On the structure of the fundamental series of generalized Harish-Chandra modules, Asian Journal of Mathematics \textbf{16} (2012), 489-514.

\vspace{.25cm}

\bibitem[V] {V} D. Vogan, The algebraic structure of the representations of semisimple Lie groups I, Ann. of Math. \textbf{109} (1979), 1-60.

\vspace{.25cm}

\bibitem[WZ] {WZ} J. Willenbring, G. Zuckerman, Small semisimple subalgebras of semisimple Lie algebras, in: Harmonic analysis, group representations, automorphic forms and invariant theory, Lecture Notes Series, Institute Mathematical Sciences, National University Singapore, \textbf{12}, World Scientific, 2007, pp. 403-429. 

\vspace{.25cm}

\bibitem[Z] {Z} G. Zuckerman, Generalized Harish-Chandra modules, in: Highlights of Lie algebraic methods, Progress in Mathematics 295, Birkhauser, 2012, pp.123-143.  
\end{thebibliography}
\end{document}